\documentclass[12pt, twoside, leqno]{article}

\usepackage{amsmath,amsthm}
\usepackage{amssymb}
\usepackage[usenames,dvipsnames]{pstricks}

\usepackage{enumerate}
\usepackage{bbm}
\usepackage{nicefrac}
\usepackage{hyperref}

\usepackage{graphicx}
\usepackage{mathrsfs}

\makeatletter
\@namedef{subjclassname@2010}{%
  \textup{2010} Mathematics Subject Classification}
\makeatother
\DeclareMathOperator{\lcm}{lcm}

\newtheorem{thm}{Theorem}[section]
\newtheorem{cor}[thm]{Corollary}
\newtheorem{lem}[thm]{Lemma}
\newtheorem{prop}[thm]{Proposition}

\theoremstyle{definition}
\newtheorem{defin}[thm]{Definition}
\newtheorem{rem}[thm]{Remark}

\numberwithin{equation}{section}

\newcommand{\raz}{\mathbbm{1}}

\newcommand{\divides}{\mid}
\newcommand{\ndivides}{\nmid}
\newcommand{\Per}{\operatorname{Per}\nolimits}
\newcommand{\un}{\underline}
\newcommand{\cm}{{\mathcal M}}
\newcommand{\ov}{\overline}
\newcommand{\Z}{{\mathbb{Z}}}
\newcommand{\N}{{\mathbb{N}}}
\newcommand{\bdelta}{\boldsymbol{\delta}}

\begin{document}


\baselineskip=17pt


\title{Automorphisms of Toeplitz $\mathscr{B}$-free systems}

\author{Aurelia Bartnicka\\
Faculty of Mathematics and Computer Science,\\
Nicolaus Copernicus University\\
Chopina 12/18\\
87-100 Toru\'{n}\\
Poland\\
E-mail: aurbart@mat.umk.pl}

\date{}

\maketitle

\renewcommand{\thefootnote}{}



\renewcommand{\thefootnote}{\arabic{footnote}}
\setcounter{footnote}{0}

\begin{abstract}
For each $\mathscr{B}$-free subshift given by $\mathscr{B}=\{2^ib_i\}_{i\in\N}$, where $\{b_i\}_{i\in\N}$ is a set of pairwise coprime odd numbers greater than one, it is shown that its automorphism group consists solely of powers of the shift.
\end{abstract}

\section{Introduction}
Consider the compact space $\{0,1\}^\Z$ of two sided sequences provided with the product topology. On this space we have the natural $\Z$-action by the \emph{left shift} $S$, i.e. \begin{equation*}\label{eq:translations}
S({(x_m)}_{m\in \Z})={(y_m)}_{m\in \Z},
\end{equation*} 
where $y_m=x_{m+1}$ for each $m\in\Z$.

We say that a set $Z\subset\{0,1\}^\Z$ is a \emph{subshift} if it is closed and invariant under the shift, i.e. $SZ=Z$. For any subshift system $(Z, S)$ by the \emph{automorphism group} $C(S)$ we mean the set of all homeomorphisms $U \colon Z\to Z$ which commute with $S$, i.e. $U\circ S=S\circ U$. The set $C(S)$ is a group.

A natural source of subshifts is to take $x\in\{0,1\}^\Z$ and $(X_x,S)$, where $X_x:=\ov{\mathcal{O}_S(x)}$ with $\mathcal{O}_S(x):=\{S^n x\ :\ n\in\Z\}$ (equivalently, $X_x=\{y\in\{0,1\}^\Z \ : \ \text{ each block appearing in }y$\\
$\text{ appears in }x\}$). Each $x=(x_n)_{n\in\Z}\in\{0,1\}^\Z$ can be identified with $\raz_{\text{supp }x}$, where $\text{supp }x:=\{n\in\Z \ : \ x_n=1\}$ is the~\emph{support} of $x$. In the present paper, we will study the case, where the support is the set of \emph{$\mathscr{B}$-free numbers} $\mathcal{F}_\mathscr{B}:=\Z\setminus\mathcal{M}_\mathscr{B}$ for some $\mathscr{B}\subset\N$, where $\mathcal{M}_\mathscr{B}$ is the set of multiples, i.e. $\mathcal{M}_\mathscr{B}=\bigcup_{b\in\mathscr{B}}b\Z$. Let $\eta=\raz_{\mathcal{F}_\mathscr{B}}$. By a \emph{$\mathscr{B}$-free subshift} we mean $(X_\eta,S)$. We will constantly assume that $\mathscr{B}$ is \emph{primitive} that is, for any $b$, $b'\in\mathscr{B}$ if $b\divides b'$ then $b=b'$.

When $(X,S)$ is a subshift the automorphism group is countable because each its member is a coding \cite{MR0259881} and it is interesting to know how complicated $C(S)$ can be; see e.g. \cite{MR3565928} and \cite{MR3436754,DDMP} for some recent results and the references therein. In the present paper, we consider the $\mathscr{B}$-free systems whose dynamical properties are under intensive study; see e.g. ~\cite{Ab-Le-Ru},~\cite{Baake:2015aa},~\cite{MR3055764}, \cite{MR3356811},~\cite{Peckner:2015aa},~\cite{sarnak-lectures} and especially \cite{AB-SK-JKP-ML} for more historical references. In case of Erd\H{o}s, i.e.\ when $\mathscr{B}$ is infinite, its elements are pairwise coprime and $\sum_{b\in\mathscr{B}}1/b<\infty$, it has been proved that the automorphism group is \emph{trivial}, i.e. it consists solely of powers of the shift, \cite{Me}. Recall that the Erd\H{o}s case implies $(X_\eta,S)$ to be proximal and non-minimal \cite{Ab-Le-Ru}. On the other hand, when a $\mathscr{B}$-free subshift $(X_\eta,S)$ is minimal then it must be Toeplitz \cite{AB-SK-JKP-ML}, as a matter of fact, $\eta$ itself has to be Toeplitz \cite{SK-GK-ML}. The main result of this paper is the following.

\begin{thm}\label{centralizer}
Let $\{b_i\}_{i\in\N}$ be a set of pairwise coprime odd numbers greater than one and $\mathscr{B}=\{2^ib_i\}_{i\in\N}$. Then the automorphism group of the $\mathscr{B}$-free subshift $(X_\eta,S)$ is trivial.
\end{thm}

While here we consider a minimal case in contrast to \cite{Me}, where a proximal case has been studied, one more difference between the result in \cite{Me} and Theorem \ref{centralizer} can be pointed out. Indeed, in a general setup one can consider all continuous (not necessarly invertible) maps commuting with $S$. Such a~semigroup of maps in the context of \cite{Me} is definitely non-trivial. Indeed, the subshifts considered in \cite{Me} are hereditary, i.e. for any $x\in X_\eta$ and $y\leq x$ coordinatewise, we have $y\in X_\eta$. Notice that for any $k\in\N$ we can extend any code $C\colon\{0,1\}^k\to\{0,1\}$ to the map of the space $\{0,1\}^\Z$ by the formula $\widehat{C}(x)(m)=C(x[m,m+k-1])$ for any $x\in\{0,1\}^\Z$ and any $m\in\Z$. Then $\widehat{C}$ is a continuous map commuting with $S$. Assume that for any $B\in\{0,1\}^k$ with $B(0)=0$ we have $C(B)=0$. Then $\widehat{C}(x)(m)=C(x[m,m+k-1])\leq x(m)$ for any $x\in\{0,1\}^\Z$ and any $m\in\Z$. So for any $x\in X_\eta$ we have $\widehat{C}(x)\leq x$ coordinatewise. Because $X_\eta$ is hereditary, we obtain $\widehat{C}(x)\in X_\eta$. Hence $\widehat{C}(X_\eta)\subset X_\eta$. In general, such maps are not invertible, for example the map $\widehat{C}\colon X_\eta\to X_\eta$ given by the code $C\colon\{0,1\}\to\{0,1\}$ defined by $C(0)=C(1)=0$ is non-invertible. From that point of view the class considered in our paper is completely different. Our examples are coalescent, i.e. all continuous maps commuting with $S$ are homeomorphisms (see Corollary \ref{coal}). We recall that automorphism groups of general Toeplitz subshifts need not be trivial; see, e.g., \cite{Bu-Kw}.

A $\mathscr{B}$-free system is called \emph{taut} if $\bdelta(\mathcal{M}_{\mathscr{B}\setminus \{b\}})<\bdelta(\mathcal{M}_\mathscr{B})$ for each $b\in\mathscr{B}$, where $\bdelta$ stands for the logarithmic density. As shown in \cite{AB-SK-JKP-ML}, the $\mathscr{B}$-free systems that are taut have interesting dynamical properties. In the last section, we show that $\mathscr{B}$-free systems which are minimal (equivalently, Toeplitz) are taut. Moreover, we formulate some open questions.
\section{Preliminaries}
\subsection{Toeplitz subshifts}
In this subsection, we recall the definition and properties of Toeplitz subshifts of the space $0$--$1$ sequences indexed by $\Z$.

We say that $x\in \{0,1\}^\Z$ is a \emph{Toeplitz sequence} whenever for any $n\in\Z$ there exists $s_n\in\N$ such that $x(n+k\cdot s_n)=x(n)$ for any $k\in\Z$. A subshift $(Z,S)$, $Z\subset \{0,1\}^\Z$ is said to be \emph{Toeplitz} if $Z=\ov{\mathcal{O}_S(y)}$ for some Toeplitz sequence $y\in \{0,1\}^\Z$.
\begin{defin}(due to S. Williams, \cite{MR 86k:54062})
Let $s\in\N$. By $\Per_s(x)$ we denote
$\{n\in\Z \ : \ \forall_{k\in\Z} \ x(n) = x(n + ks)\}$. (If $s'\divides s$ then $\Per_{s'}(x)\subset \Per_s(x)$.)
By an \emph{essential period} of $x$ we mean $s$ for which $\Per_{s'} (x)\neq \Per_s(x)\neq\emptyset$ for any positive integer $s'< s$. Finally, a \emph{periodic structure} of $x$ is any sequence $s = (s_m)_{m\in \N}$ of essential periods such that $s_m\divides s_{m+1}$ for each $m\in\N$ and
\begin{equation}\label{period}
\bigcup_{m\in\N} \Per_{s_m}(x) = \Z.
\end{equation}
\end{defin}
Every Toeplitz sequence has a periodic structure, which is not unique because any subsequence of a periodic structure is a periodic structure too.
As in \cite{MR2180227}, we can put $s_m$ equals the least common multiple of the minimal periods of the positions in $[-m, m]$.
\begin{defin}(see \cite{MR41426})
A Toeplitz sequence $x\in\{0,1\}^\Z$ is \emph{regular} if there exists a periodic structure $(s_m)_{m\in\N}$ of $x$ such that 
$$\lim_{m\to\infty}\frac{|\Per_{s_m}(x)\cap[0,s_m)|}{s_m}=1.$$
\end{defin}
Let $(p_t)_{t=1}^\infty$ be a periodic structure of a Toeplitz sequence $x\in\{0,1\}^\Z$ and $G$ be the inverse limit group, $G=\varprojlim\Z/p_t\Z$. That is,
$$G=\{(n_t)_{t=1}^\infty \ : \ n_t\in\Z/p_t\Z\text{ and }n_{t+1}\equiv n_t \ (\bmod \ p_t)\text{ for any }t\geq1\}.$$ 
Notice that $G$ is metrizable by the metric $$
|(n_t)_{t=1}^\infty,(n_t')_{t=1}^\infty)|=\max\left\{\frac{1}{i+1} \ : \ n_i\neq n_i'\right\},$$
for any $(n_t)_{t=1}^\infty, (n_t')_{t=1}^\infty\in G$. We denote $n(1,1,1,\ldots)$ by $\overline{n}$ for any $n\in\Z$. Let $T$ be the translation of $G$ by the unit element $\overline{1}$. Then $G$ is a~compact monothetic group with generator $\overline{1}$. In \cite{MR 86k:54062}, it is proved that the system $(G,T)$ is the \emph{maximal equicontinuous factor} of $(\ov{\mathcal{O}_S(x)},S)$, i.e. the system $(G,T)$ is the largest system such that the family of maps $\{T^n \ : \ n\in\Z\}$ is equicontinuous and there exists a continuous surjective $\pi\colon\ov{\mathcal{O}_S(x)}\to G$ such that $\pi \circ T=S\circ \pi$ ($\pi$ is called a \emph{factor map}). In other words, the system $(G,T)$ is the largest equicontinuous factor of $(\ov{\mathcal{O}_S(x)},S)$ (any other equicontinuous factor of $(\ov{\mathcal{O}_S(x)},S)$ is a~factor of $(G,T)$). Every topological dynamical system has the maximal equicontinuous factor which is unique up to isomorphism (see, e.g., \cite{MR1958753}).

Let $A_t\in\{0,1,\_\}^{p_t}$ be a block such that
\begin{equation*}
A_t(n)=\begin{cases} 
x(n), \text{ if }x(n)=x(n+kp_t) \text{ for each } k\in\Z, \\
\_, \text{ otherwise},
\end{cases}
\end{equation*}
for any $n\in\{0,1,\ldots,p_t-1\}$.
By a \emph{filled place} in $A_t$ we mean each $i\in\{0,1,\ldots,p_t-1\}$ such that $A_t(i)\in\{0,1\}$. We call the symbol $\_$ a~\emph{hole}. By the \emph{$p_t$-skeleton} of $x$ we will mean a~sequence obtained from $x$ by replacing $x(n)$ by a hole for all $n\not\in \Per_{p_t}(x)$.
If $x\in\{0,1\}^\Z$ is a~regular Toeplitz sequence then the sequence of blocks $(A_t)_{t=1}^\infty$ satisfies the following conditions:
\begin{enumerate}[(A)]
\item the block $A_{t+1}$ is obtained as a concatenation $A_t A_t\ldots A_t$, where some holes are filled by symbols $0$ or $1$,\label{cond1}
\item $\lim\limits_{t\to\infty}\nicefrac{r_t}{p_t}=1$, where $r_t$ is the number of filled places in $A_t$,\label{cond2}
\item for every $i\in\N$ there exists an index $t$ such that $A_t(i)\in\{0,1\}$.\label{cond3}
\end{enumerate}
For any given $t\geq1$ let $\{n\ : \ A_t(n)=\_\}=\{I_1^{(t)}<I_2^{(t)}<\ldots<I_{s_t}^{(t)}\}$ be the set of all positions of holes in $A_t$. 
\begin{defin}(see \cite{Bu-Kw})
We say that a Toeplitz sequence $x\in\{0,1\}^\Z$ has property $(Sh)$ (separated holes) if \begin{equation}\label{min}
k_t:=\min(\{I_{j+1}^{(t)}-I_j^{(t)}, j=1,2,\ldots,s_t-1\}\cup\{p_t-I_{s_t}^{(t)}+I_1^{(t)}\})
\end{equation} tends to $\infty$ as $t\to\infty$.
\end{defin}
As is mentioned in \cite{Bu-Kw}, each sequence of blocks $(A_t)_{t=1}^\infty$ satisfying \eqref{cond1}-\eqref{cond3} determines a~Toeplitz sequence $x$, which may be periodic.

\begin{rem}
In the case of Toeplitz sequences satisfying condition $(Sh)$, each continuous $U\colon \ov{\mathcal{O}_S(x)}\to\ov{\mathcal{O}_S(x)}$ which commutes with $S$ is a~homeomorphism. In other words, such Toeplitz subshifts are topologically coalescent (see Proposition 3 in \cite{Bu-Kw}).
\end{rem}
By a \emph{$t$-symbol} of $x$ we mean any block $A$ of length
$p_t$ such that
$A = x[\ell p_t, \ell p_t +p_t - 1]$ for some $\ell\in\Z$. For each $t$-symbol of $x$ we have $\{n\in\{0,1,\ldots,p_t\} \ : \ A(n)=A_t(n)\}=\{0,1,\ldots,p_t\}\setminus\{I_1^{(t)},I_2^{(t)},\ldots,I_{s_t}^{(t)}\}$.

Let $g=(n_t)_{t=1}^\infty\in G$. We denote by $A_t(g)$ the following block
$$
A_t(g)=A_tA_t[n_t,n_t+p_t-1].$$
Let $\{ J_1^{(t)}(g)<J_2^{(t)}(g)<\ldots<J_{s_t}^{(t)}(g)\}$ be the set of all holes in $A_t(g)$.
The sequence $(A_t(g))_{t=1}^\infty$ satisfies conditions \eqref{cond1} and \eqref{cond2} so it determines a~two-sided sequence $x(g)\in\{0,1,\_\}^\Z$ such that for any $t\geq1$ and any $0\leq i<p_t$ satisfying $A_t(g)(i)\in\{0,1\}$ we have $$
x(g)(i+\ell p_t)=A_t(g)(i),$$
for all $\ell\in\Z$. Each of the $t$-symbols of $x(g)$ coincides with $A_tA_t[n_t,n_t+p_t-1]$ except at the places $J_1^{(t)}(g), J_2^{(t)}(g),\ldots,J_{s_t}^{(t)}(g)$. By the definition of $A_t(g)$, we obtain $\{J_k^{(t)}(g)\ : \ k=1,2\ldots, s_t\}=\{I_k^{(t)}-n_t \ (\bmod \ p_t) \ : \ k=1,2,\ldots, s_t\}$. 

Let $G_0=\{g\in G \ : \ x(g) \text{ is a $0$--$1$ Toeplitz sequence}\}$ and $G_2=\{g=(n_t)_{t=1}^\infty\in G \ : \ A_t(n_t)=\_ \ \text{ for each } t\geq0\}$. Then $G_0=G\setminus G_1$,
where $G_1=G_2+\Z\overline{1}$. Moreover, $G_0$ is of Haar measure one (see \cite[p. 48]{Bu-Kw}).

Let $\pi$ be a factor map from $(\ov{\mathcal{O}_S(x)},S)$ to $(G,T)$, defined as in \cite{MR 86k:54062},
$\pi^{-1}(\{g\})=\{y\in\ov{\mathcal{O}_S(x)} \ : \ y\text{ has the same } p_t\text{-skeleton as }S^{n_t}x$ for any $t\geq1\} \text{ for any }g=(n_t)_{t=1}^\infty\in G$. Assume that a Toeplitz sequence $x\in\{0,1\}^\Z$ has property $(Sh)$. Then $\pi^{-1}(\{g\})$ contains at most two elements for any $g\in G$ (see Remark $4$ in \cite{Bu-Kw}) and $\pi^{-1}(\{g\})=\{x(g)\}$ for any $g\in G_0$ (see Remark $2$ in \cite{Bu-Kw}). 

If $U\in C(S)$ then $U$ induces a continuous map $U'$ on $G$ commuting with $T$. Indeed, notice that $\overline{0}\in G_0$ because $x(\overline{0})=x$. Put 
\begin{equation}\label{U}
U'T^n\overline{0}=\pi US^nx
\end{equation} for any $n\in\Z$. Because $\pi$ is a factor map and $U\in C(S)$, we have $\pi US^nx=\pi S^nUx=T^n\pi Ux=\overline{n}+\pi Ux=\overline{n}+U'\overline{0}$. Hence $U'\colon\{T^n\overline{0} \ : \ n\in\Z\}\to G$ is the restriction of the translation by $U'\overline{0}$. Because $G$ is a monothetic group, $U'$ is the translation by $g_0:=U'\overline{0}$. In this case it is natural to say that $g_0$ can be lifted to $U$. 
Notice that if $g_0\in G$ can be lifted to $U$, then $U$ is unique. Indeed, let $U_1\neq U_2\in C(S)$ induce a map $U'\in C(T)$. Then because $(\ov{\mathcal{O}_S(x)}, S)$ is minimal, $U_1y\neq U_2y$ for any $y\in\ov{\mathcal{O}_S(x)}$. Notice that $U'\pi=\pi U_1=\pi U_2$. So $U'g\in G_1$ for any $g\in G$ but $G_1$ is of Haar measure zero. This contradicts that $U'$ is the translation. The question arises which elements $g_0\in G$ can be lifted to elements of $C(S)$. 
\subsection{General lemmas}
In this subsection, we include the general facts about arithmetic progressions and $\mathscr{B}$-free systems which will be used later.
\begin{lem}\label{kogruencja}
Let $b\in\Z$ and $a,m\in\Z\setminus\{0\}$. Then a congruence 
\begin{equation}\label{kong}
ax\equiv b \ (\bmod \ m)
\end{equation}
has a solution $x\in\Z$ if and only if $\gcd(a,m)\divides b$.
\end{lem}
\begin{proof}
Let $x\in\Z$ be a solution of \eqref{kong}. Then $\gcd(a,m)\divides ax$ and $\gcd(a,m)\divides \ell m$ for any $\ell\in\Z$. Hence $\gcd(a,m)\divides b$ because $b$ is a difference of $ax$ and some multiple of $m$.

Assume that $\gcd(a,m)\divides b$. By the Euclid's Extended Algorithm, there exists $k\in\Z$ such that $ka\equiv\gcd(a,m)\ (\bmod \ m)$. Then 
$$ka\frac{b}{\gcd(a,m)}\equiv b \ (\bmod \ m).$$
So, $x=\frac{kb}{\gcd(a,m)}$ is a solution of \eqref{kong}. This completes the proof.
\end{proof}
In \cite{AB-SK-JKP-ML}, the characterization of periodicity of the characteristic function of $\mathscr{B}$-free numbers is given by the following lemma 
\begin{lem}[see Proposition 4.25 in \cite{AB-SK-JKP-ML}]\label{periodic}
Let $\mathscr{B}\subset \N$ be primitive, i.e. for any $b,b'\in\mathscr{B}$ if $b\divides b'$ then $b=b'$. Then $\mathscr{B}$ is finite if and only if $\eta$ is
periodic, with the minimal period $\lcm(\mathscr{B})$.
\end{lem}
\begin{lem}\label{inter}
Let $a,b,r\in\N$. If $\gcd(a,b)\divides r$ then 
\begin{equation}\label{equal}
(a\Z+r)\cap b\Z=\lcm(a,b)\Z+bs
\end{equation}
for some $s\in\Z$ such that $bs\equiv r \ (\bmod \ a)$. Moreover, if $\gcd(a,b)\ndivides r$ then $(a\Z+r)\cap b\Z=\emptyset$.
\end{lem}
\begin{proof}
Assume that $\gcd(a,b)\divides r$. Then by Lemma \ref{kogruencja}, there exists $k\in\Z$ such that $ak+r\equiv 0 \ (\bmod \ b)$. Hence there exists $s\in\Z$ such that $ak+r=bs$. Then $\lcm(a,b)\Z+bs\subset(a\Z+r)\cap b\Z$. Let $x\in(a\Z+r)\cap b\Z$. Then there exist $k',s'\in\Z$ such that $x=bs'=ak'+r$. Moreover, $x-bs\in b\Z\subset \lcm(a,b)\Z$ and $x-(ak+r)=a(k-k')\in a\Z\subset\lcm(a,b)\Z$. Hence $x\in\lcm(a,b)\Z+bs$. Thus, equality \eqref{equal} holds.

Assume that $\gcd(a,b)\ndivides r$ and $x\in(a\Z+r)\cap b\Z$. Notice that $b\divides x$ and $\gcd(a,b)\divides a$. Hence $\gcd(a,b)\divides r$. But this is a contradiction. The assertion follows.
\end{proof}
\begin{lem}\label{intersect}
Let $a,r,m\in\N$. Then $|(a\Z+r)\cap[0,ma)|=m$.
\end{lem}
\begin{proof}
Note that for any non-negative integer $i$ we have 
\begin{align*}
&n\in(a\Z+r)\cap[ia,(i+1)a)\iff ia\leq n=as+r<(i+1)a \\
&\text{ for some } s\in\Z.
\end{align*}
Without loss of generality, we can assume that $0\leq r<a$ because $a\Z+r=a\Z+(r \bmod a)$.
Notice that $ai+r\in(a\Z+r)\cap[ia,(i+1)a)$.
Assume that $n=as+r\in(a\Z+r)\cap[ia,(i+1)a)$. Then $s$ is an integer satisfying the following inequalities
$$i-\frac{r}{a}\leq s < i+1-\frac{r}{a}.$$
 Then $0\leq\frac{r}{a}<1$, so $i-1<s\leq i$. This implies that the intersecion $(a\Z+r)\cap[ia,(i+1)a)$ has exactly one element for any $0\leq i< m$. The interval $[0,ma)$ is the disjoint union of $[ia,(i+1)a)$ for $0\leq i< m$. Hence the assertion holds.
\end{proof}

\section{Subfamily of Toeplitz $\mathscr{B}$-free systems}
Let $\mathscr{B}=\{2^ib_i\}_{i\in\N}$, where $\{b_i\}_{i\in\N}$ is a set of pairwise coprime odd numbers greater than $1$. We say that $n\in\Z$ is \emph{$\mathscr{B}$-free number} if $n\not\in\bigcup\limits_{i\in\N}2^ib_i\Z$. By $\mathcal{F}_{\mathscr{B}}$ we will denote the set of all $\mathscr{B}$-free numbers and by $\eta$ its characteristic function. This means that 
\begin{equation}\label{Z:2}
\eta(n)=\begin{cases}
1,& \text{if }n \text{ is }\mathscr{B}\text{-free},\\
0,& \text{otherwise}.
\end{cases}
\end{equation}
for any $n\in\Z$.

Let $p_t=2^tb_1b_2\ldots b_t$ for any $t\in\N$. We will prove that $\eta$ is a~Toeplitz sequence and $(p_t)_{t\in\N}$ is an example of its periodic structure.
\begin{lem}[see Example 3.1. in \cite{AB-SK-JKP-ML}] \label{T}
The sequence $\eta$ is a Toeplitz sequence.
\end{lem}
\begin{proof}
Let $n\in\Z$. 
\begin{itemize}
\item If $\eta(n)=0$ then there exists $j\in\N$ such that $2^jb_j\divides n$. Notice that $2^jb_j\divides n+p_j\ell=2^jb_j(\frac{n}{2^jb_j}+b_1b_2\ldots b_{j-1}\ell)$ for any $\ell\in\Z$. So $\eta(n+p_j\ell)=0$ for any $\ell\in\Z$. Hence $n\in \Per_{p_j}(\eta)$.
\item If $\eta(n)=1$ then 
\begin{equation}\label{eta1}
\eta(n+p_{a+1}\ell)=1\text{ for any }\ell\in\Z,
\end{equation}
where $a$ is a non-negative integer and $m$ is an odd integer such that $n=2^am$. Suppose that \eqref{eta1} does not hold. So for some $j\in\N$, we have
\begin{equation}\label{e2}
2^jb_j\divides n+p_{a+1}\ell\text{ for some }\ell\in\Z.
\end{equation}
Notice $2^{a+1}\ndivides n+p_{a+1}\ell=2^a(m+2b_1b_2\ldots b_{a+1}\ell)$. So $j\leq a$ and by~\eqref{e2}, we have $2^jb_j \divides n$, which contradicts $\eta(n)=1$. Hence \eqref{eta1} holds. This implies $n\in \Per_{p_{a+1}}(\eta)$. 
\end{itemize}
The assertion follows.
\end{proof} 
\begin{lem}\label{essential}
For any $t\in\N$ the number $p_t$ is an essential period of $\eta$.
\end{lem}
\begin{proof}
Let $s<p_t$ be a positive integer, $m$ be odd and $a$ be a non-negative integer such that $s=m2^a$. Three cases appear:
\begin{enumerate}[(I)]
\item there exists $1\leq i\leq t-1$ such that $\gcd(b_i,s)<b_i$,\label{pierwszy}
\item $\gcd(b_j,s)=b_j$ for all $1\leq j\leq t-1$ and $\gcd(b_t,s)<b_t$,\label{drugi}
\item $\gcd(b_j,s)=b_j$ for all $1\leq j\leq t$.\label{trzeci}
\end{enumerate}
Assume that condition \eqref{pierwszy} holds. Let $b_i'=\gcd(b_i,s)$ and $\ell\in\Z$. We claim that 
\begin{equation}\label{eta}
\eta(2^{t-1}b_i'+p_t\ell)=1.
\end{equation}
Suppose not, so that for some $j\in\N$, we have $2^jb_j\divides 2^{t-1}b_i'+p_t\ell$. Because $b_i'$ is odd, we have $2^t\ndivides 2^{t-1}b_i'+p_t\ell=2^{t-1}(b_i'+2b_1b_2\ldots b_t\ell)$. Hence $j<t$ and because $b_j$ is an odd number greater than $1$, we obtain $b_j\divides b_i'+2b_1b_2\ldots b_t\ell$ which implies $b_j\divides b_i'\divides b_i$. Because $\{b_r\}_{r\in\N}$ is a set of pairwise coprime integers, $j=i$ but then $b_i=b_i'$ which contradicts condition \eqref{pierwszy}. Hence equality \eqref{eta} holds. Moreover, we have 
\begin{equation}\label{podziel}
 2^{\min(a,i)}b_i'=\gcd(2^ib_i,s)\divides 2^{t-1}b_i'.
\end{equation} 
Indeed, notice that $2^i$ and $b_i$ are coprime. Hence we have $\gcd(2^ib_i,s)=\gcd(2^i,s)\gcd(b_i,s)=2^{\min(a,i)}b_i'$. Because $i\leq t-1$, we have $\min(a,i)\leq t-1$. This implies that \eqref{podziel} holds.
By Lemma \ref{kogruencja} and \eqref{podziel}, there exists $n\in\Z$ such that $2^{t-1}b_i'+sn\equiv 0 \ (\bmod \ 2^ib_i)$, so $\eta(2^{t-1}b_i'+sn)=0$. Hence $2^{t-1}b_i'\in \Per_{p_t}(\eta)$ but $2^{t-1}b_i'\notin \Per_s(\eta)$.

Assume that condition \eqref{drugi} holds. Then 
\begin{equation}\label{es}
\eta(2^tb_t+2^{t+1}s)=1.
\end{equation}
Suppose not, so there exists $j\in\N$ such that $2^jb_j\divides 2^tb_t+2^{t+1}s$. Because $b_t$ is odd, we have $2^{t+1}\ndivides2^tb_t+2^{t+1}s=2^t(b_t+2s)$. Hence $j\leq t$ and because $b_j$ is an odd number greater than 1, we obtain $b_j\divides b_t+2s$. If $j<t$ then by condition \eqref{drugi}, we have $b_j\divides s$ which implies $b_j\divides b_t$. But this is impossible because $\{b_r\}_{r\in\N}$ is a set of pairwise coprime integers. So $j=t$ and $b_t\divides 2s$. But this contradicts condition \eqref{drugi}. Hence equation \eqref{es} holds. Moreover, note $2^tb_t\divides 2^tb_t+p_t\ell=2^tb_t(1+b_1b_2\ldots b_{t-1}\ell)$ for any $\ell\in\Z$. So $2^tb_t\in \Per_{p_t}(\eta)$ but $2^tb_t\notin \Per_s(\eta)$. 

Assume that condition \eqref{trzeci} holds. Then $a<t$ because $b_1,b_2,\ldots, b_t$ are pairwise coprime and $s<p_t=2^tb_1b_2\ldots b_t$. We claim that 
\begin{equation}\label{eta2}
\eta(2^tb_t+s)=1.
\end{equation}
Suppose not, so there exists $j\in\N$ such that $2^jb_j\divides2^tb_t+s$. Notice $2^{a+1}\ndivides2^tb_t+s=2^a(2^{t-a}b_t+m)$. Hence $j\leq a<t$. Because $b_j$ is odd, we obtain $b_j\divides2^{t-a}b_t+m$ which together with \eqref{trzeci} implies $b_j\divides b_t$. This is impossible because $b_j$ and $b_t$ are coprime. So equality \eqref{eta2} holds. Hence $2^tb_t\notin \Per_s(\eta)$ but $2^tb_t\in \Per_{p_t}(\eta)$.

By above $p_t$ is an essential period of $\eta$.
\end{proof}
\begin{lem}\label{structure}
The sequence $\eta$ is a Toeplitz sequence with a periodic structure $(p_t)_{t\geq1}$.
\end{lem}
\begin{proof}
Let $t\in\N$. Then by Lemma \ref{essential}, $p_t$ is an essential period. Notice that $p_{t+1}=2p_tb_{t+1}$. So $p_t\divides p_{t+1}$. By arguments using in the proof of Lemma \ref{T}, equality \eqref{period} holds.
\end{proof}
\begin{lem}
The Toeplitz sequence $\eta$ is not periodic.
\end{lem}
\begin{proof}
Because $\{b_i\}_{i\in\N}$ is a set of pairwise coprime odd numbers, we obtain $\mathscr{B}$ is primitive and infinite. By Lemma \ref{periodic}, $\eta$ is not periodic.
\end{proof}
Now we will give the characterization of places where $A_t$ has holes and compute the number of holes.
\begin{lem}\label{holes}
Any $0\leq s<p_t$ is a hole in $A_t$ if and only if $s$ satisfies the following conditions
\begin{equation}\label{hole1}
2^t\divides s\text{ and }b_i\ndivides s\text{ for any }1\leq i\leq t.
\end{equation}
The number of holes in $A_t$ equals 
\begin{equation}\label{number}
\prod\limits_{i=1}^t\left(b_i-1\right)=:s_t.
\end{equation}
\end{lem}
\begin{proof}
Let $0\leq s< p_t$. Let $m$ be odd and $a$ be a non-negative integer such that $s=m2^a$. Three cases appear:
\begin{enumerate}[(i)]
\item $2^jb_j\divides s$ for some $1\leq j\leq t$,\label{second}
\item $2^t\ndivides s$ and $2^ib_i\ndivides s$ for any $1\leq i\leq t$,\label{first}
\item $2^t\divides s\text{ and }2^ib_i\ndivides s\text{ for any }1\leq i\leq t$.\label{hole}
\end{enumerate}
Assume that condition \eqref{second} holds. Then $2^jb_j\divides s+p_t\ell=2^jb_j(\frac{s}{2^jb_j}+2^{t-j}\frac{b_1b_2\ldots b_t}{b_j}\ell)$ for any $\ell\in\Z$. So $s\in \Per_{p_t}(\eta)$.

Assume that condition \eqref{first} holds. 
Then \begin{equation}\label{ndiv}
2^ib_i\ndivides s+p_t \ell\text{ for any }i\in\N\text{ and }\ell\in\Z.
\end{equation} Indeed, suppose that $2^ib_i\divides s+p_t \ell$ for some $\ell\in\Z$ and some $i\in\N$. Notice $2^{a+1}\ndivides s+p_t \ell=2^a(m+2^{t-a}b_1b_2\ldots b_t\ell)$. So $i\leq a<t$ which implies $2^ib_i\divides p_t$. Hence $2^ib_i\divides s$, which contradicts \eqref{first}. So \eqref{ndiv} holds. This implies $s\in \Per_{p_t}(\eta)$.

Assume that condition \eqref{hole} holds. Notice that \eqref{hole} is equivalent to $2^t\divides s$ and $b_i\ndivides s$ for any$1\leq i\leq t$. Indeed, if \eqref{hole} holds then $2^i\divides s$ for any $1\leq i\leq t$. We claim that $s\notin \Per_{p_t}(\eta)$. Indeed, we have two possibilities
\begin{itemize}
\item $\eta(s)=1$. Then notice that $2^t=\gcd(2^{t+1}b_{t+1},p_t)\divides s$. So by Lemma \ref{kogruencja},
there exists $n\in\Z$ such that $p_tn+s\equiv 0 \ (\bmod \ 2^{t+1}b_{t+1})$. Hence $\eta(s+p_tn)=0$. 
\item $\eta(s)=0$. If $a=t$ then $2^{t+1}\ndivides s+2p_t=2^t(m+2b_1b_2\ldots b_t)$. Suppose that 
\begin{equation}\label{divide}
2^jb_j\divides s+2p_t\text{ for some }j\in\N.
\end{equation}
Then $j\leq t$. Because $b_j$ is odd, we have $b_j\divides m+2b_1b_2\ldots b_t$, which implies $b_j\divides m$. But this contradicts \eqref{hole}. So \eqref{divide} does not hold. If $a>t$ then because $b_1,b_2,\ldots,b_t$ are odd, we obtain $2^{t+1}\ndivides s+p_t=2^t(2^{a-t}m+b_1b_2\ldots b_t)$. Suppose that 
\begin{equation}\label{divide2}
2^jb_j\divides s+p_t\text{ for some }j\in\N.
\end{equation} 
Then $j\leq t$. Because $2^jb_j\divides p_t$, we have $2^jb_j\divides s$. But this contradicts \eqref{hole}. So \eqref{divide2} does not hold. Hence we obtain $\eta(s+p_t)=1$.
\end{itemize}
There are $b_1b_2\ldots b_t$ non-negative multiples of $2^t$ smaller than $p_t$. We need to know how many of them are not multiples of any $b_i$ for $1\leq i\leq t$. We will compute how many of non-negative multiples of $2^t$ smaller than $p_t$ are multiples of $b_i$ for some $1\leq i\leq t$. This is equivalent to compute the power of the union $\bigcup_{i=1}^t B_i$, where $B_i=\{0\leq w<b_1b_2\ldots b_t\ : \ b_i\divides w\}$ for any $1\leq i\leq t$. Notice that because $\{b_i\}_{i\in\N}$ are pairwise coprime, for any $\emptyset\neq J\subset\{1,2,\ldots,t\}$ we have $\bigcap_{j\in J}B_j=\bigcap_{j\in J} b_j\Z\cap [0,b_1b_2\ldots b_t)=(\prod_{j\in J} b_j)\Z\cap [0,b_1b_2\ldots b_t)$. By Lemma \ref{intersect}, we obtain 
\begin{equation}\label{bi}
|\bigcap_{j\in J}B_j|=\frac{b_1b_2\ldots b_t}{\prod_{j\in J} b_j}.
\end{equation}
By the Inclusion-Exclusion Principle and \eqref{bi}, we have that 
\begin{align*}
&|\bigcup_{i=1}^t B_i|=\sum_{\emptyset\neq J\subset\{1,2,\ldots,t\}}(-1)^{|J|-1}|\bigcap_{j\in J}B_j|=\sum_{\emptyset\neq J\subset\{1,2,\ldots,t\}}(-1)^{|J|-1}\frac{b_1b_2\ldots b_t}{\prod_{j\in J} b_j}=\\
&b_1b_2\ldots b_t\sum_{\emptyset\neq J\subset\{1,2,\ldots,t\}}(-1)^{|J|-1}\frac{1}{\prod_{j\in J} b_j}.
\end{align*}
Hence the number of $0\leq s<p_t$ satisfying \eqref{hole} equals
\begin{align*}
&b_1b_2\ldots b_t\left(1-\sum_{\emptyset\neq J\subset\{1,2,\ldots,t\}}(-1)^{|J|-1}\frac{1}{\prod_{j\in J} b_j}\right)=b_1b_2\ldots b_t\prod\limits_{i=1}^t\left(1-\frac{1}{b_i}\right)\\
&=\prod\limits_{i=1}^t\left(b_i-1\right).
\end{align*}
This completes the proof.
\end{proof}
\begin{cor}[see Remark 3.2 in \cite{AB-SK-JKP-ML}]
The Toeplitz sequence $\eta$ is regular. 
\end{cor}
\begin{proof}
By Lemma \ref{holes}, we have $$\lim_{t\to\infty}\frac{|(\Z\setminus\Per_{p_t}(\eta))\cap[0,p_t)|}{p_t}=\lim_{t\to\infty}\frac{\prod\limits_{i=1}^t\left(b_i-1\right)}{2^tb_1b_2\ldots b_t}\leq\lim_{t\to\infty}\frac{1}{2^t}=0.$$ The assertion follows.
\end{proof}
\begin{cor}\label{consp} 
We have $k_t=2^t$(see \eqref{min}).
\end{cor}
\begin{proof}
By conditions \eqref{hole1} and $2^t\divides p_t$, we have $2^t\divides k_t$. Notice that for any $t\geq1$ the pair of the smallest numbers satisfying \eqref{hole1} is $2^t$ and $2^{t+1}$. So $I_1^{(t)}=2^t$ and $I_2^{(t)}=2^{t+1}$. Hence $k_t=2^t$.
\end{proof}
\begin{lem}\label{Sh}
The point $\eta$ has property $(Sh)$.
\end{lem}
\begin{proof}
Let $t\geq1$. As we proved above the block $A_t$ has holes at some non-zero multiples of $2^t$ (see \eqref{hole1}). Moreover, $2^t\divides p_t$. Hence the distance between consecutive holes in $A_t$ is at least $2^t$. So $\eta$ has separated holes. 
\end{proof}
\begin{cor}\label{coal}
The system $(X_\eta,S)$ is coalescent, i.e. each continuos map $U\colon X_\eta\to X_\eta$ commuting with $S$ is a homeomorphism.
\end{cor}
\begin{proof}
By Proposition 3 in \cite{Bu-Kw}, any Toeplitz system $\ov{\mathcal{O}_S(x)}$ such that $x$ has property $(Sh)$, is coalescent. Hence by Lemma \ref{Sh}, the system $(X_\eta,S)$ is coalescent.
\end{proof}
In the following lemmas, we will give specific properties of holes in $A_t$.

For any $t\in\N$ let $I_1^{(t)}< I_2^{(t)}<\ldots< I_{s_t}^{(t)}$ be all positions of the holes in the block $A_t$ and for any $h\in G$ let $J_1^{(t)}(h)<J_2^{(t)}(h)<\ldots<J_{s_t}^{(t)}(h)$ be all positions of the holes in the block $A_t(h)$.
\begin{lem}\label{podz}
Let $t\geq1$. Then for any $1\leq i\leq t$ we have $$\left\{\frac{I_1^{(t)}}{2^t},\frac{I_2^{(t)}}{2^t},\ldots,\frac{I_{s_t}^{(t)}}{2^t}\right\} \bmod b_i=\{1,2,\ldots, b_i-1\}.
$$
\end{lem}
\begin{proof}
Let $1\leq i\leq t$ and $1\leq j\leq b_i-1$. Put $$A:=(b_i\Z+j)\cap[0,b_1b_2\ldots b_t)\cap\bigcup\limits_{m=1}^tb_m\Z$$
 and $$A_F:=(b_i\Z+j)\cap\bigcap_{n\in F}b_n\Z\text{ for any }F\subset\{1,2,\ldots,t\}\setminus\{i\}.$$ Because $b_1,b_2,\ldots,b_t$ are pairwise coprime, we have $$\bigcap_{n\in F}b_n\Z=\left(\prod_{n\in F}b_n\right)\Z.$$  By Lemma \ref{inter}, for any $F\subset\{1,2,\ldots,t\}\setminus\{i\}$, we obtain
\begin{equation}
A_F=\left(\prod_{n\in F\cup\{i\}}b_n\right)\Z+\left(\prod_{n\in F}b_n\right)s,
\end{equation}
where $s\in\Z$ satisfies $\left(\prod\limits_{n\in F}b_n\right)s\equiv j \ (\bmod \ b_i)$.
By Lemma \ref{intersect}, the intersection $A_F\cap[0,b_1b_2\ldots b_t)$ has $\frac{b_1b_2\ldots b_t}{\prod\limits_{n\in F\cup\{i\}}b_n}$ elements. By the Inclusion-Exclusion Principle (using the same arguments as in the proof of identity \eqref{number}), we obtain that the number of elements of $A$ equals
\begin{align*}
&\sum_{1\leq n\neq i\leq t}\frac{b_1b_2\ldots b_t}{b_ib_n}-\sum_{\substack{1\leq n\neq i\neq n'\leq t\\
n\neq n'}}\frac{b_1b_2\ldots b_t}{b_ib_nb_{n'}}+\ldots+(-1)^t\frac{b_1b_2\ldots b_t}{\prod\limits_{n=1}^tb_n}\\
&=\frac{b_1b_2\ldots b_t}{b_i}\left(1-\prod_{1\leq n\neq i\leq t}\left(1-\frac{1}{b_n}\right)\right).
\end{align*}
Hence the number of elements of $A$ is smaller than the number of elements of the set $(b_i\Z+j)\cap[0,b_1b_2\ldots b_t)$, which, by Lemma \ref{intersect}, equals $\frac{b_1b_2\ldots b_t}{b_i}$. Hence there exists $m\in(b_i\Z+j)\cap[0,b_1b_2\ldots b_t)$ such that $b_n\ndivides m$ for any $1\leq n\leq t$. So, by \eqref{hole1}, we have $m\in\left\{\frac{I_1^{(t)}}{2^t},\frac{I_2^{(t)}}{2^t},\ldots,\frac{I_{s_t}^{(t)}}{2^t}\right\}$.
The assertion follows.
\end{proof}
\begin{lem}\label{space}
Let $h=(n_t)_{t=1}^\infty\in G$ can be lifted to $U\in C(S)$. Then there exists $t_0\geq1$ such that for any $t\geq t_0$
$$I_1^{(t)}-J_1^{(t)}(h)=I_2^{(t)}-J_2^{(t)}(h)=\ldots=I_{s_t}^{(t)}-J_{s_t}^{(t)}(h)=:k',$$
where $k'$ depends on $t_0$.
\end{lem}
\begin{proof}
Let $h=(n_t)_{t=1}^\infty\in G$ can be lifted to $U\in C(S)$. By Curtis-Hedlund-Lyndon Theorem (see Theorem 3.4 in \cite{MR0259881}), there exist $I\in\Z$ and $k\in\N$ and a function $f\colon\{0,1\}^k\to\{0,1\}$ such that
\begin{equation}
U(y)(m)=f(y[m+I,m+I+k-1])
\end{equation}
for any $m\in\Z$ and any $y\in\ov{\mathcal{O}_S(\eta)}$. Without loss of generality we can assume that $I=0$. Indeed, notice that $U(S^Iy)(m)=f(y[m,m+k-1])$ for any $m\in\Z$ and any $y\in\ov{\mathcal{O}_S(\eta)}$. Moreover, $US^I\in C(S)$ if and only if $U\in C(S)$.

Let $t_0\geq1$ be large enough so the distance between consecutive holes in $A_{t_0}$ be greater than $k$, i.e. $2^{t_0}>k$ (see Corollary \ref{consp}).
 
Let $t\geq t_0$. By the definition, $\eta(h)$ has the same $p_t$-skeleton as $S^{n_t}\eta$. So $A_t(h)$ has $s_t$ holes and the space between consecutive holes is divisible by $2^t$. We claim that for any $0\leq j<s_t$ there exists unique $0\leq i<s_t$ such that 
\begin{equation}\label{ineq}
0\leq I_i^{(t)}-J_j^{(t)}(h)<k.
\end{equation}
Suppose not, then for any $\ell\in\Z$ we use $k$ places from $\eta$, which are $p_t$-periodic, to code $J_j^{(t)}(h)+p_t\ell$ by $f$ but this contradicts that $J_j^{(t)}(h)$ is a hole. We only need to show uniqueness of $i$. Suppose that $0\leq I_i^{(t)}-J_j^{(t)}(h)<k$ and $0\leq I_{i'}^{(t)}-J_j^{(t)}(h)<k$ for $1\leq i',i<s_t$ then $|I_i^{(t)}-I_{i'}^{(t)}|<k<2^t$. So $i=i'$. 

By \eqref{ineq}, there exists unique $1\leq i\leq s_t$ such that $I_i^{(t)}-J_1^{(t)}(h)< k$ (look at the picture).
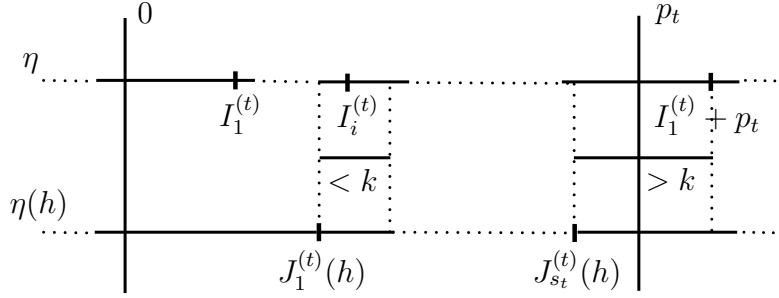
\begin{center}
{
\begin{pspicture}(0,-1.92)(10.219078,1.92)
\psline[linecolor=black, linewidth=0.04](1.5333333,1.74)(1.5333333,-1.92)
\psline[linecolor=black, linewidth=0.04](1.15,0.91)(3.26,0.91)
\psline[linecolor=black, linewidth=0.04](1.13,-1.11)(4.08,-1.11)
\rput[bl](0.17,1.01){$\eta$}
\rput[bl](0.0,-0.96){$\eta(h)$}
\rput(1.8,1.81){$0$}
\rput(3.07,0.46){$I_1^{(t)}$}
\rput(4.63,0.45){$I_i^{(t)}$}
\rput(4.15,-1.61){$J_1^{(t)}(h)$}
\rput(7.54,-1.61){$J_{s_t}^{(t)}(h)$}
\psline[linecolor=black, linewidth=0.04](8.363334,1.77)(8.363334,-1.89)
\rput(8.76,1.79){$p_t$}
\rput(9.27,0.45){$I_1^{(t)}+p_t$}
\psframe[linecolor=black, linewidth=0.04, dimen=outer](3.03,1.03)(2.97,0.79)
\psframe[linecolor=black, linewidth=0.04, dimen=outer](4.52,1.0)(4.46,0.76)
\psframe[linecolor=black, linewidth=0.04, dimen=outer](4.14,-0.99)(4.08,-1.23)
\psframe[linecolor=black, linewidth=0.04, dimen=outer](9.35,1.01)(9.29,0.77)
\psframe[linecolor=black, linewidth=0.04, dimen=outer](7.54,-1.0)(7.48,-1.24)
\psline[linecolor=black, linewidth=0.04, linestyle=dotted, dotsep=0.10583334cm](4.113333,0.89)(4.113333,-1.06)
\psline[linecolor=black, linewidth=0.04, linestyle=dotted, dotsep=0.10583334cm](5.0533333,0.83)(5.0533333,-1.12)
\psline[linecolor=black, linewidth=0.04](4.13,-1.11)(5.11,-1.11)
\psline[linecolor=black, linewidth=0.04, linestyle=dotted, dotsep=0.10583334cm](5.08,-1.11)(7.4,-1.12)
\psline[linecolor=black, linewidth=0.04](7.55,-1.11)(9.67,-1.11)
\psline[linecolor=black, linewidth=0.04, linestyle=dotted, dotsep=0.10583334cm](9.74,-1.11)(10.3,-1.12)
\psline[linecolor=black, linewidth=0.04, linestyle=dotted, dotsep=0.10583334cm](9.76,0.91)(10.32,0.9)
\psline[linecolor=black, linewidth=0.04](4.1,0.89)(5.31,0.89)
\psline[linecolor=black, linewidth=0.04, linestyle=dotted, dotsep=0.10583334cm](5.4,0.88)(7.43,0.87)
\psline[linecolor=black, linewidth=0.04](7.34,0.89)(9.7,0.89)
\psline[linecolor=black, linewidth=0.04, linestyle=dotted, dotsep=0.10583334cm](3.36,0.9)(3.97,0.89)
\psline[linecolor=black, linewidth=0.04, linestyle=dotted, dotsep=0.10583334cm](7.5033336,0.88)(7.5033336,-1.07)
\psline[linecolor=black, linewidth=0.04, linestyle=dotted, dotsep=0.10583334cm](9.333333,0.85)(9.333333,-1.1)
\psline[linecolor=black, linewidth=0.04](4.12,-0.12)(5.06,-0.12)
\rput(4.56,-0.42){$<k$}
\psline[linecolor=black, linewidth=0.04](7.5,-0.12)(9.35,-0.12)
\rput(8.79,-0.4){$>k$}
\psline[linecolor=black, linewidth=0.04, linestyle=dotted, dotsep=0.10583334cm](0.43,0.91)(1.04,0.9)
\psline[linecolor=black, linewidth=0.04, linestyle=dotted, dotsep=0.10583334cm](0.43,-1.11)(1.04,-1.12)
\end{pspicture}
}
\end{center}

 If $i>1$ then because $(I_1^{(t)}+p_t)-J_{s_t}^{(t)}(h)$ is greater than $k$, there exist $1\leq j<j'\leq s_t$ and $1<i'\leq s_t$ such that $$I_{i'}^{(t)}-J_j^{(t)}(h)< k\text{ and }I_{i'}^{(t)}-J_{j'}^{(t)}(h)< k.$$ Then $J_{j'}^{(t)}(h)-J_{j}^{(t)}(h)< k$ but the distance between holes in $A_t(h)$ is at least $2^t$ which is greater than $k$. So $i=1$. By using the same argument for $i=2,3\ldots, s_t$, we can obtain that $\eta$ and $\eta(h)$ look like in the following picture
 
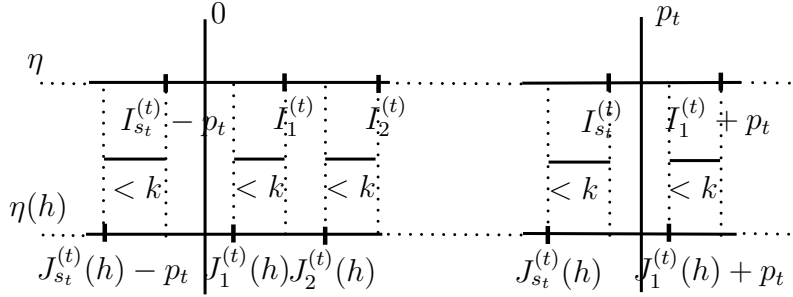
\begin{center}
{
\begin{pspicture}(0,-1.925)(10.35489,1.925)
\psline[linecolor=black, linewidth=0.04](2.6033332,1.735)(2.6033332,-1.925)
\psline[linecolor=black, linewidth=0.04](1.06,0.895)(5.06,0.895)
\psline[linecolor=black, linewidth=0.04](1.02,-1.125)(4.96,-1.125)
\rput[bl](0.24,1.005){$\eta$}
\rput[bl](0.0,-0.985){$\eta(h)$}
\rput(2.78,1.815){$0$}
\rput(3.79,0.445){$I_1^{(t)}$}
\rput(3.15,-1.555){$J_1^{(t)}(h)$}
\psframe[linecolor=black, linewidth=0.04, dimen=outer](3.69,1.015)(3.63,0.775)
\psframe[linecolor=black, linewidth=0.04, dimen=outer](4.94,0.995)(4.88,0.755)
\psline[linecolor=black, linewidth=0.04, linestyle=dotted, dotsep=0.10583334cm](5.02,-1.125)(6.82,-1.135)
\psline[linecolor=black, linewidth=0.04, linestyle=dotted, dotsep=0.10583334cm](5.14,0.885)(6.82,0.875)
\psframe[linecolor=black, linewidth=0.04, dimen=outer](2.11,1.005)(2.05,0.765)
\rput(2.18,0.435){$I_{s_t}^{(t)}-p_t$}
\rput(1.38,-1.555){$J_{s_t}^{(t)}(h)-p_t$}
\psframe[linecolor=black, linewidth=0.04, dimen=outer](1.3,-0.985)(1.24,-1.225)
\rput(5.01,0.445){$I_2^{(t)}$}
\psline[linecolor=black, linewidth=0.04, linestyle=dotted, dotsep=0.10583334cm](1.2533333,0.815)(1.2533333,-1.135)
\psline[linecolor=black, linewidth=0.04, linestyle=dotted, dotsep=0.10583334cm](2.0833333,0.785)(2.0833333,-1.165)
\psline[linecolor=black, linewidth=0.04](1.26,-0.125)(2.08,-0.125)
\rput(1.7,-0.495){$<k$}
\psline[linecolor=black, linewidth=0.04, linestyle=dotted, dotsep=0.10583334cm](0.41,0.885)(1.02,0.875)
\psline[linecolor=black, linewidth=0.04, linestyle=dotted, dotsep=0.10583334cm](0.43,-1.135)(1.0,-1.145)
\psframe[linecolor=black, linewidth=0.04, dimen=outer](3.01,-0.995)(2.95,-1.235)
\psline[linecolor=black, linewidth=0.04, linestyle=dotted, dotsep=0.10583334cm](2.9833333,0.885)(2.9833333,-1.065)
\psline[linecolor=black, linewidth=0.04, linestyle=dotted, dotsep=0.10583334cm](3.6733334,0.815)(3.6733334,-1.135)
\psline[linecolor=black, linewidth=0.04](2.99,-0.125)(3.66,-0.125)
\rput(3.33,-0.485){$<k$}
\psframe[linecolor=black, linewidth=0.04, dimen=outer](4.23,-1.025)(4.17,-1.265)
\psline[linecolor=black, linewidth=0.04, linestyle=dotted, dotsep=0.10583334cm](4.2033334,0.855)(4.2033334,-1.095)
\rput(4.29,-1.585){$J_2^{(t)}(h)$}
\psline[linecolor=black, linewidth=0.04, linestyle=dotted, dotsep=0.10583334cm](4.903333,0.835)(4.903333,-1.115)
\psline[linecolor=black, linewidth=0.04](4.2,-0.125)(4.87,-0.125)
\rput(4.54,-0.485){$<k$}
\rput(7.3,-1.595){$J_{s_t}^{(t)}(h)$}
\psline[linecolor=black, linewidth=0.04](8.403334,1.765)(8.403334,-1.895)
\rput(8.77,1.775){$p_t$}
\rput(9.42,0.435){$I_1^{(t)}+p_t$}
\psframe[linecolor=black, linewidth=0.04, dimen=outer](9.49,1.025)(9.43,0.785)
\psframe[linecolor=black, linewidth=0.04, dimen=outer](7.19,-1.015)(7.13,-1.255)
\psline[linecolor=black, linewidth=0.04](6.83,-1.115)(9.67,-1.125)
\psline[linecolor=black, linewidth=0.04, linestyle=dotted, dotsep=0.10583334cm](9.7,0.895)(10.39,0.885)
\psline[linecolor=black, linewidth=0.04](6.82,0.885)(9.73,0.895)
\rput(7.88,0.415){$I_{s_t}^{(t)}$}
\psframe[linecolor=black, linewidth=0.04, dimen=outer](8.0,1.015)(7.94,0.775)
\rput(9.3,-1.575){$J_1^{(t)}(h)+p_t$}
\psframe[linecolor=black, linewidth=0.04, dimen=outer](8.8,-0.985)(8.74,-1.225)
\psline[linecolor=black, linewidth=0.04, linestyle=dotted, dotsep=0.10583334cm](9.75,-1.115)(10.38,-1.125)
\psline[linecolor=black, linewidth=0.04, linestyle=dotted, dotsep=0.10583334cm](7.983333,0.795)(7.983333,-1.155)
\psline[linecolor=black, linewidth=0.04, linestyle=dotted, dotsep=0.10583334cm](7.1633334,0.825)(7.1633334,-1.125)
\rput(7.5,-0.515){$<k$}
\psline[linecolor=black, linewidth=0.04](7.17,-0.165)(7.99,-0.165)
\psline[linecolor=black, linewidth=0.04, linestyle=dotted, dotsep=0.10583334cm](8.773334,0.865)(8.773334,-1.085)
\psline[linecolor=black, linewidth=0.04, linestyle=dotted, dotsep=0.10583334cm](9.463333,0.795)(9.463333,-1.155)
\psline[linecolor=black, linewidth=0.04](8.78,-0.145)(9.45,-0.145)
\rput(9.09,-0.505){$<k$}
\end{pspicture}
}

\end{center}

Let $k'=I_1^{(t_0)}-J_1^{(t_0)}(h)$.
There exists $m\in\N$ such that $J_2^{(t_0)}(h)=J_1^{(t_0)}(h)+2^{t_0}m$. Notice that $0\leq I_2^{(t_0)}-J_2^{(t_0)}(h)=I_1^{(t_0)}-J_1^{(t_0)}(h)+2^{t_0}\frac{I_2^{(t_0)}-I_1^{(t_0)}}{2^{t_0}}$ $-2^{t_0}m=k'+2^{t_0}\left(\frac{I_2^{(t_0)}-I_1^{(t_0)}}{2^{t_0}}-m\right)<2^{t_0}$. So $\frac{I_2^{(t_0)}-I_1^{(t_0)}}{2^{t_0}}-m=0$, which implies $I_2^{(t_0)}-J_2^{(t_0)}(h)=k'$. We can repeat the same argument for any $j=3,4,\ldots,s_{t_0}$ and obtain $I_j^{(t_0)}-J_j^{(t_0)}(h)=k'$ for any $j=3,4\ldots,s_{t_0}$.

Let $t=t_0+1$. Then by the definitions of blocks $A_t$ and $A_t(h)$ (see \eqref{cond1}), we have the following inclusions
\begin{align*}
&\{I_1^{(t)},I_2^{(t)},\ldots,I_{s_t}^{(t)}\}\subset\bigcup_{j=0}^{2b_t-1}\left(\{I_1^{(t_0)},I_2^{(t_0)},\ldots,I_{s_{t_0}}^{(t_0)}\}+jp_{t_0}\right) 
\end{align*}
and
\begin{align*}
&\{J_1^{(t)}(h),J_2^{(t)}(h),\ldots,J_{s_t}^{(t)}(h)\}\subset\bigcup_{j=0}^{2b_t-1}\left(\{J_1^{(t_0)}(h),J_2^{(t_0)}(h),\ldots,J_{s_{t_0}}^{(t_0)}(h)\}+jp_{t_0}\right). 
\end{align*}
Hence for any $1\leq i\leq s_t$ there exist $1\leq i_0\leq s_{t_0}$ and $0\leq j<2b_t$ such that $I_i^{(t)}=I_{i_0}^{(t_0)}+jp_{t_0}$ and $J_i^{(t)}(h)=J_{i_0}^{(t_0)}(h)+jp_{t_0}$. Indeed, $I_i^{(t)}-J_i^{(t)}(h)<k<2^{t_0}$. So $I_i^{(t)}-J_i^{(t)}(h)=I_{i_0}^{(t_0)}-J_{i_0}^{(t_0)}(h)=k'$.

We can repeat above arguments for any $t>t_0+1$.
\end{proof}
\begin{rem}
In above proof, we repeated some arguments contained in proofs of Proposition 3 and Theorem 1 in \cite{Bu-Kw}.
\end{rem}
\begin{lem}\label{shift}
Let $h=(n_t)_{t=1}^\infty\in G$ can be lifted to $U\in C(S)$. Then there exists $t_0\geq1$ such that for any $t\geq t_0$ we have
\begin{align*}
&(\{I_1^{(t)},I_2^{(t)},\ldots,I_{s_t}^{(t)}\}-n_t+k') \bmod p_t=\{I_1^{(t)},I_2^{(t)},\ldots,I_{s_t}^{(t)}\}.
\end{align*}
\end{lem}
\begin{proof}
By the definition of $\eta(h)$, we have
\begin{align*}
&\{J_1^{(t)}(h),J_2^{(t)}(h),\ldots,J_{s_t}^{(t)}(h)\}=(\{I_1^{(t)},I_2^{(t)},\ldots,I_{s_t}^{(t)}\}-n_t)\ \bmod p_t.
\end{align*}
By Lemma \ref{space}, we obtain for any $t\geq t_0$
$$\{J_1^{(t)}(h),J_2^{(t)}(h),\ldots,J_{s_t}^{(t)}(h)\}=\{I_1^{(t)},I_2^{(t)},\ldots,I_{s_t}^{(t)}\}-k'.$$
Notice that 
\begin{align*}
&((\{I_1^{(t)},I_2^{(t)},\ldots,I_{s_t}^{(t)}\}-n_t) \bmod p_t+k')\bmod p_t=\\
&(\{I_1^{(t)},I_2^{(t)},\ldots,I_{s_t}^{(t)}\}-n_t+k')\bmod p_t.
\end{align*}
The assertion follows.
\end{proof}
\begin{lem}\label{eq}
Let $h=(n_t)_{t=1}^\infty\in G$ can be lifted to $U\in C(S)$. Then there exists $t_0\geq1$ such that for any $t\geq t_0$
\begin{equation}
2^t\divides n_t-k'.
\end{equation}
\end{lem}
\begin{proof}
By Lemma \ref{shift}, for any $t\geq t_0$ there exists $1\leq m\leq s_t$ such that
$$I_1^{(t)}-n_t\equiv I_m^{(t)}-k'\ (\bmod \ p_t).$$
Hence because $2^t\divides p_t$ and $2^t\divides I_{m}^{(t)}-I_1^{(t)}$, we obtain $2^t\divides n_t-k'$.
\end{proof}
\section{Proof of Theorem \ref{centralizer}}
Assume that $h=(n_t)_{t=1}^\infty\in G\setminus \Z\overline{1}$ can be lifted to $U\in C(S)$.
Let $t\geq t_0$ be large enough to satisfy the conclusion of Lemma \ref{space} and $0<n_t-k'$. Such a $t$ exists because the coordinates of $h$ are unbounded.
We claim
\begin{equation}\label{div}
b_i\divides n_t-k'\text{ for each }1\leq i\leq t.
\end{equation}
Indeed, by Lemma \ref{eq}, we have $2^t\divides n_t-k'$. Assume that $\frac{n_t-k'}{2^t}\equiv j \ (\bmod \ b_i)$ for some $1\leq i\leq t$ and some $1\leq j\leq b_i-1$. Then by Lemma \ref{podz}, there exists $1\leq m\leq s_t$ such that $\frac{I_m^{(t)}}{2^t}\equiv j \ (\bmod \ b_i)$. So $b_i\divides \frac{I_m^{(t)}-n_t+k'}{2^t}$, which implies $b_i\divides I_m^{(t)}-n_t+k'$. Because $b_i\divides p_t$, we obtain $b_i\divides (I_m^{(t)}-n_t+k') \bmod p_t$. By Lemma \ref{shift}, we have $(I_m^{(t)}-n_t+k') \bmod p_t\in\{I_1^{(t)},I_2^{(t)},\ldots,I_{s_t}^{(t)}\}$. But this contradicts \eqref{hole1}. So \eqref{div} holds. Hence $n_t-k'$ is a multiple of $b_1,b_2,\ldots,b_t$ and $2^t$, which are coprime. Because $0\leq n_t,k'<p_t$, so $n_t-k'=0$. But this contradicts $n_t-k'>0$. The assertion follows.
\section{Final remarks}
\begin{rem}
There exist regular Toeplitz sequences satisfying condition $(Sh)$ whose automorphism groups are non-trivial. There are some of \emph{$k_t$-Toeplitz sequences}. For these sequences, the block $A_t$ has $k_t$ holes with equal distances between them. For more information see the chapter \emph{$k_t$-Toeplitz flows} in \cite{Bu-Kw}. 
In the case of $\eta$, $I_1^{(t)}=2^t$ and $I_2^{(t)}=2^{t+1}$ for any $t\geq1$, so if the holes of $A_t$ were equidistant, their number would equal $\frac{p_t}{2^t}=b_1b_2\ldots b_t$ but by Lemma \ref{holes}, the number of holes in $A_t$ is smaller than $b_1b_2\ldots b_t$. Hence, $\eta$ is not a~$k_t$-Toeplitz sequence. 
\end{rem}
\begin{rem}\footnote{Such examples was pointed out to me by W. Bu{\l}atek.}
We will give an example of a Toeplitz sequence $x\in\{0,1\}^\Z$ for which the block $A_t$ has two holes for each $t\geq1$ and the Boolean complement $\neg$ is an element of the automorphism group, recall that $(\neg y)(n)=0$ if and only if $y(n)=1$ for any $y\in\overline{\mathcal{O}_S(x)}$. Notice that $\neg$ is generated by the code $f\colon\{0,1\}\to\{0,1\}$ of length $1$ defined by the formula $f(0)=1$ and $f(1)=0$. Because $\neg$ is continuous and invertible for the full shift, to prove $\neg$ is an element of the automorphism group of the Toeplitz subshift, it is sufficient to show that $\neg(\overline{\mathcal{O}_S(x)})\subset\overline{\mathcal{O}_S(x)}$. Let $n\in\N$ and $B_1\in\{0,1\}^n$. By $\widetilde{B}\in\{0,1\}^n$ we denote the Boolean complement of the block $B\in\{0,1\}^n$, i.e. $\widetilde{B}(i)=f(B(i))$ for any $0\leq i<n$. Let $A_1=B_1\_\widetilde{B_1}\_$ and $(c_t)_{t\in\N}\subset\{0,1\}$. Then put $B_2:=B_1c_1\widetilde{B_1}c_1B_1$ and $A_2:=B_2 \_ \widetilde{B_2}\_$. Notice that $A_2$ is the concatenation $A_1A_1A_1$, where all holes except of two are filled by symbols $0$ or $1$. Let $B_{t+1}=B_tc_t\widetilde{B_t}c_tB_t$ and $A_{t+1}=B_{t+1}\_ \widetilde{B_{t+1}}\_ $ for any $t\in\N$. Notice that the sequence of blocks $(A_t)_{t=1}^\infty$ satisfies \eqref{cond1}-\eqref{cond3}, so it determines a Toeplitz sequence $x\in\{0,1\}^\Z$. By the definition of $x$, it is non-periodic and for any block $B$ appearing in $x$ its Boolean complement $\widetilde{B}$ appears in $x$. So $\neg(\overline{\mathcal{O}_S(x)})\subset\overline{\mathcal{O}_S(x)}$. 
\end{rem}
\begin{rem}\footnote{This fact was pointed out to me by S. Kasjan who also gave a proof of it which we recall.}
We will show that the Boolean complement is not an element of the automorphism group of $\mathscr{B}$-free systems. In fact, we will prove that the Boolean complement is an element of the automorphism group of $\mathscr{B}$-free systems if and only if $\mathscr{B}=\{2\}$. We will consider two cases:
\begin{enumerate}[(1)]
\item there exist $b,c\in\mathscr{B}$ such that $\gcd(b,c)=1$,\label{1}
\item $\gcd(b,c)>1$ for any $b,c\in\mathscr{B}$.\label{2}
\end{enumerate} 
Assume that \eqref{1} holds. First notice that $\eta(bi)=0$ for any $i\in\Z$. Suppose that $\neg\in C(S)$. Then for any $n\in\N$ there exists $r\in\Z$ such that 
\begin{equation}\label{3}
\eta(bi+r)=1 \text{ for any } i=1,2,\ldots,n.
\end{equation} 
Let $n\geq c$ and $r\in\Z$ satisfies \eqref{3}. By Lemma \ref{kogruencja}, the congruence 
\begin{equation}\label{cong}
bx+r\equiv 0 \bmod c
\end{equation} 
has a solution $x_0\in\Z$. Note that also $x_0+kc$ is a solution of \eqref{cong} for any $k\in\Z$. Hence $\eta(b(x_0+kc)+r)=0$ for any $k\in\Z$. Notice that there exists $k\in\Z$ such that $1\leq x_0+kc\leq c\leq n$. But by \eqref{3}, we have $\eta(b(x_0+kc)+r)=1$. This is a contradiction.

Assume that \eqref{2} holds. Notice that if $2\in\mathscr{B}\neq\{2\}$ then the block $00$ appears in $\eta$. So there exist $n\in\Z$ and $b,c\in\mathscr{B}$ such that $b\divides n$ and $c\divides n+1$. But then $\gcd(b,c)=1$. Hence $2\not\in\mathscr{B}$. Then $\min\mathscr{B}\geq3$. So the block $11$ appears in $\eta$. Suppose that $\neg\in C(S)$. Then $00$ apears in $\eta$. So there exist $b,c\in\mathscr{B}$ such that $\gcd(b,c)=1$, which contradicts \eqref{2}. Notice that in case of $\mathscr{B}=\{2\}$ we have $S=\neg$. Hence $\neg$ is an element of the automorphism group of the $\mathscr{B}$-free subshift if and only if $\mathscr{B}=\{2\}$. 
\end{rem}
For $A\subset \N$, we recall several notions of asymptotic density. We have:
\begin{itemize}
\item
$\underline{d}(A):=\liminf\limits_{N\to\infty}\frac1N|A\cap[1,N]| \text{ (\emph{lower density} of $A$)}$,
\item
$\overline{d}(A):=\limsup\limits_{N\to\infty}\frac1N|A\cap[1,N]| \text{ (\emph{upper density} of $A$)}$.
\end{itemize}
If the lower and the upper density of $A$ coincide, their common value $d(A):=\un{d}(A)=\ov{d}(A)$ is called the \emph{density} of~$A$. 

The \emph{logarithmic density} of $A$ is 
$$\bdelta(A):=\lim\limits_{N\to\infty}\frac{1}{\log N}\sum\limits_{a\in A, 1\leq a\leq N}\frac{1}{a},$$
whenever the limit exists.
A set $\mathscr{B}\subset \N\setminus \{1\}$ is called \emph{Behrend} if $d(\cm_{\mathscr{B}})=1$ (see \cite{Davenport1936,MR0043835}). A set $\mathscr{B}$ is \emph{taut} when is primitive, i.e. for any $b,b'\in\mathscr{B}$ if $b\divides b'$ then $b=b'$, and does not contain $c\mathscr{A}$ with $c\in\N$ and $\mathscr{A}\subset \N\setminus\{1\}$ that is Behrend. In~\cite{MR1414678}, it is proved that
for any $\mathscr{B}\subset\N$ there exists a primitive $\mathscr{B}'\subset\mathscr{B}$ such that $\mathcal{M}_{\mathscr{B}}=\mathcal{M}_{\mathscr{B}'}$.
As shown in \cite{AB-SK-JKP-ML}, we have a ``good'' theory of $\mathscr{B}$-free subshifts, both for topological dynamics as well as ergodic theory point of view whenever $\mathscr{B}$ is taut. On the other hand, Toeplitz $\mathscr{B}$-free systems $(X_\eta,S)$ play a special role in the theory of $\mathscr{B}$-free systems. The Toeplitz case can be characterized by the minimality of $(X_\eta,S)$, \cite{AB-SK-JKP-ML}, more precisely by the fact that $\eta$ itself is a Toeplitz sequence \cite{SK-GK-ML}. We have the following.
\begin{prop}\label{taut}
For any primitive $\mathscr{B}\subset \N\setminus\{1\}$, if $\eta=\raz_{\mathcal{F}_\mathscr{B}}$ is a~Toeplitz sequence, then $\mathscr{B}$ is taut.
\end{prop} 
\begin{proof}
Assume that $\raz_{\mathcal{F}_\mathscr{B}}$ is a~Toeplitz sequence and $\mathscr{B}$ is not taut. Then for some $c\in\N$ and $\mathscr{A}\subset\N\setminus\{1\}$ such that $d(\mathcal{M}_{\mathscr{A}})=1$ we have $c\mathscr{A}\subset\mathscr{B}$. So 
\begin{equation}\label{density}
\frac{1}{c}=d(c\mathcal{M}_\mathscr{A})\leq \underline{d}(\mathcal{M}_\mathscr{B}\cap c\Z)\leq\underline{d}(c\Z)=\frac{1}{c}
\end{equation}
Notice that $c'\not\in\mathscr{B}$ for any $c'|c$ because $\mathscr{B}$ is primitive. So $\eta(c)=1$. By the assumption that $\eta$ is a Toeplitz sequence, there exists $m\in\N$ such that 
\begin{equation*}
\eta(c+m\ell)=1\text{ for any }\ell\in\Z.
\end{equation*} 
Hence $c+m\Z\subset\mathcal{F}_\mathscr{B}$, which implies $(c+m\Z)\cap c\Z\subset\mathcal{F}_\mathscr{B}$. By Lemma \ref{inter}, we have 
\begin{equation*}
(c+m\Z)\cap c\Z=\lcm(c,m)\Z+c. 
\end{equation*}
 So $\underline{d}(\mathcal{F}_\mathscr{B}\cap c\Z)\geq\frac{1}{\lcm(c,m)}$ which contradicts \eqref{density}. The assertion follows.
\end{proof}
\begin{rem}
The sequences we have considered here are regular Toeplitz sequences, hence the systems they determine are strictly ergodic \cite{MR41426} and of zero entropy. There exist non-regular Toeplitz sequences of $\mathscr{B}$-free origin; see Example 4.2 in \cite{SK-GK-ML}. 
\end{rem}
There seem to be natural questions: to characterize those Toeplitz sequences which are of $\mathscr{B}$-free origin. Which non-regular Toeplitz sequences are of $\mathscr{B}$-free origin? Can entropy be positive in the Toeplitz $\mathscr{B}$-free case? Is always the automorphism group trivial (question by Lema\'{n}czyk for all $\mathscr{B}$-free systems)?

\subsection*{Acknowledgements}
This research was supported by Narodowe Centrum Nauki grant UMO-2014/15/B/ST1/03736. The author thanks her advisor Mariusz Lema{{\'n}}czyk for helpful discussions, remarks and motivating to improve this text and Mieczysław Mentzen for some useful comments after reading a preliminary version.

\end{document}